\documentclass[a4paper]{amsart}

\usepackage[margin=3.24cm]{geometry}

\usepackage{amsmath,amssymb,amsthm}
\usepackage[utf8]{inputenc}
\usepackage[T1]{fontenc}
\usepackage{libertine}
\usepackage[libertine,cmintegrals,cmbraces]{newtxmath}
\usepackage{enumitem}
\usepackage{microtype}
\usepackage{color}
\definecolor{darkred}{RGB}{139,0,0}
\definecolor{darkblue}{RGB}{0,0,139}
\definecolor{darkgreen}{RGB}{0,100,0}
\usepackage{hyperref}
\hypersetup{
  colorlinks   = true, 
  urlcolor     = darkred, 
  linkcolor    = darkred, 
  citecolor   = darkgreen}
\usepackage{tikz}
\usepackage{tikz-cd}
\usepackage[capitalise]{cleveref}

\setenumerate[1]{label=(\roman*),}

\AtBeginDocument{%
   \def\MR#1{}}

\newcommand{\id}{\ensuremath{\operatorname{id}}}
\newcommand{\im}{\ensuremath{\operatorname{im}}}

\newcommand{\SO}{\ensuremath{\operatorname{SO}}}

\newcommand{\BO}{\ensuremath{\operatorname{BO}}}

\newcommand{\KO}{\ensuremath{\mathbf{KO}}}

\newcommand{\Diffuo}{\ensuremath{\operatorname{Diff}}}
\newcommand{\BlockDiffuo}{\ensuremath{\widetilde{\operatorname{Diff}}}}

\newcommand{\BDiffuo}{\ensuremath{\operatorname{BDiff}}}

\newcommand{\BlockBDiffuo}{\ensuremath{\operatorname{B\widetilde{Diff}}}}

\newcommand{\hAutuo}{\ensuremath{\operatorname{hAut}}}
\newcommand{\BhAutuo}{\ensuremath{\operatorname{BhAut}}}

\newcommand{\hAutWg}{\ensuremath{\operatorname{hAut}_\partial(W_{g,1})}}
\newcommand{\BhAutWg}{\ensuremath{\operatorname{BhAut}_\partial(W_{g,1})}}
\newcommand{\DiffWg}{\ensuremath{\operatorname{Diff}_\partial(W_{g,1})}}
\newcommand{\BDiffWg}{\ensuremath{\operatorname{BDiff}_\partial(W_{g,1})}}
\newcommand{\BlockDiffWg}{\ensuremath{\operatorname{\widetilde{Diff}}_\partial(W_{g,1})}}
\newcommand{\BlockBDiffWg}{\ensuremath{\operatorname{B\widetilde{Diff}}_\partial(W_{g,1})}}

\newcommand{\BhAutWgone}{\ensuremath{\operatorname{BhAut}_\partial(W_{g+1,1})}}

\newcommand{\BDiffWgone}{\ensuremath{\operatorname{BDiff}_\partial(W_{g+1,1})}}

\newcommand{\BlockBDiffWgone}{\ensuremath{\operatorname{B\widetilde{Diff}}_\partial(W_{g+1,1})}}

\newcommand{\interior}[1]{\ensuremath{\operatorname{int}(#1)}}

\newcommand{\catsingle}[1]{\ensuremath{\mathcal{#1}}}

\newcommand{\oH}{\ensuremath{\operatorname{H}}}
\newcommand{\oO}{\ensuremath{\operatorname{O}}}
\newcommand{\oB}{\ensuremath{\operatorname{B}}}

\newcommand{\bfC}{\ensuremath{\mathbf{C}}}

\newcommand{\bfG}{\ensuremath{\mathbf{G}}}

\newcommand{\bfL}{\ensuremath{\mathbf{L}}}

\newcommand{\bfO}{\ensuremath{\mathbf{O}}}

\newcommand{\bfZ}{\ensuremath{\mathbf{Z}}}
\newcommand{\bfQ}{\ensuremath{\mathbf{Q}}}

\newcommand{\bfSp}{\ensuremath{\mathbf{Sp}}}

\newcommand{\fra}{\ensuremath{\mathfrak{a}}}
\newcommand{\frg}{\ensuremath{\mathfrak{g}}}

\newcommand{\cB}{\ensuremath{\catsingle{B}}}

\newcommand{\cL}{\ensuremath{\catsingle{L}}}

\newcommand{\CE}{\ensuremath{\operatorname{CE}}}

\newcommand{\ra}{\rightarrow}
\newcommand{\lra}{\longrightarrow}

\newcommand{\xlra}[1]{\overset{#1}{\longrightarrow}}

\newcommand{\Sp}{\operatorname{Sp}}
\newcommand{\GL}{\operatorname{GL}}

\newcommand{\Der}{\operatorname{Der}}

\newcommand{\Mod}[1]{\ (\mathrm{mod}\ #1)}

\newtheorem{bigthm}{Theorem}
\newtheorem{bigcor}[bigthm]{Corollary}

\newtheorem{thm}{Theorem}[section]
\newtheorem*{nthm}{Theorem}
\newtheorem{lem}[thm]{Lemma}

\theoremstyle{definition}

\theoremstyle{remark}

\newtheorem*{nrem}{Remark}

\begin{document}

\title[A note on rational homological stability for automorphisms of manifolds]{A note on rational homological stability for\\ automorphisms of manifolds}
\author{Manuel Krannich}
\email{krannich@dpmms.cam.ac.uk}
\address{Centre for Mathematical Sciences, Wilberforce Road, Cambridge CB3 0WB, UK}
\begin{abstract}
By work of Berglund and Madsen, the rings of rational characteristic classes of fibrations and smooth block bundles with fibre $D^{2n}\sharp(S^n\times S^n)^{\sharp g}$, relative to the boundary, are for $2n\ge 6$ independent of $g$ in degrees $*\le (g-6)/2$. In this note, we explain how this range can be improved to $*\le g-2$ using cohomological vanishing results due to Borel and classical invariant theory. This implies that the analogous ring for smooth bundles is independent of $g$ in the same range, provided the degree is small compared to the dimension.
\end{abstract}

\maketitle

The classical approach to the homotopy type of the diffeomorphism group $\Diffuo_\partial(M)$ of a manifold relative to its boundary is by comparison to the space $\BlockDiffuo_\partial(M)$ of block diffeomorphisms, since their difference is approximated by Waldhausen's algebraic $K$-theory of spaces. The space $\BlockDiffuo_\partial(M)$ in turn can be understood in terms of the more accessible space of homotopy equivalences $\hAutuo_\partial(M)$ via surgery theory. Assuming that $M$ is of dimension $2n$ and $\partial M$ is nonempty, there are stabilisation maps  
\begin{equation*}
  \begin{gathered}
  \BDiffuo_\partial(M)\lra\BDiffuo_\partial(M\sharp (S^n\times S^n))\\
\BlockBDiffuo_\partial(M)\lra\BlockBDiffuo_\partial(M\sharp (S^n\times S^n))\\
\BhAutuo_\partial(M)\lra\BhAutuo_\partial(M\sharp (S^n\times S^n))
\end{gathered}
\end{equation*}
induced by extending automorphisms by the identity. In \cite{BerglundMadsenI}, Berglund--Madsen proved that for $2n\ge6$ and $M$ being $W_{g,1}=D^{2n}\sharp(S^n\times S^n)^{\sharp g}$, the upper two maps induce rational homology isomorphisms in degrees $*\le\min(n-3,(g-6)/2)$. Independently, Galatius--Randal-Williams \cite{GRWold} showed via a different approach that for $2n\ge6$, the stabilisation map for $\BDiffuo_\partial(W_{g,1})$ induces an isomorphism in integral homology in the range $*\le(g-4)/2$, independent of the dimension. Their method has been extended later \cite{Friedrich,GRWII,GRWI,PerlmutterI,PerlmutterII} to obtain more refined stability results for $\BDiffuo_\partial(M)$ for a wide class of manifolds. In the case of block diffeomorphisms, Berglund--Madsen improved upon their prior results in their impressive work \cite{BerglundMadsen}, inter alia, by removing the dependence of the range on the dimension and by proving a similar result for homotopy automorphisms.

\begin{nthm}[Berglund--Madsen]For $2n\ge4$, the stabilisation map
\[ \oH_*(\BhAutWg;\bfQ)\lra\oH_*(\BhAutWgone;\bfQ)\]
is an isomorphism for $*\le (g-6)/2$ and an epimorphism for $*\le (g-4)/2$. For $2n\ge 6$, the map 
\[\oH_*(\BlockBDiffWg;\bfQ)\lra\oH_*(\BlockBDiffWgone;\bfQ)\] has the same property.\end{nthm}

One step in the proof relies on a stability result for certain unitary groups due to Charney \cite{Charney}, which is responsible for the specific ranges in the theorem. The main purpose of this note is to explain an improvement of these ranges, obtained primarily by replacing Charney's result with a combination of classical invariant theory and a quantitative version of Borel's vanishing result on the cohomology of arithmetic groups \cite{BorelI,BorelII}.

\begin{bigthm}\label{mainthm} For $2n\ge4$, the stabilisation map
\[ \oH_*(\BhAutWg;\bfQ)\lra\oH_*(\BhAutWgone;\bfQ)\]
is an isomorphism for $*\le g-2+c$, where $c=1$ if $n$ is odd and $c=0$ if $n$ is even. For $2n\ge 6$, the map 
\[\oH_*(\BlockBDiffWg;\bfQ)\lra\oH_*(\BlockBDiffWgone;\bfQ)\] has the same property.
\end{bigthm}

\begin{nrem}
For $n\neq3,7$ odd, \cref{mainthm} gives the best possible isomorphism range of \emph{slope} $1$, i.e.\,a range of the form $*\le g+k$ for a constant $k$. This is because for such $n$, the groups $\oH_1(\BhAutWg;\bfQ)$ and $\oH_1(\BlockBDiffWg;\bfQ)$ are $1$-dimensional for $g=1$ and trivial otherwise, which can be derived for instance from a combination of \cite[Prop.\,5.3, Ex.\,5.5]{BerglundMadsen} with \cite[Lem.\,A.1, A.2]{Krannich}.
\end{nrem}

\begin{nrem}
Building on Berglund--Madsen's argument, Grey \cite{Grey} established rational homological stability results for the spaces of homotopy equivalences and block diffeomorphisms of connected sums of products of spheres of possibly different dimensions. An improvement of his results similar to \cref{mainthm} would require a slope $1$ range for Borel's vanishing results for the rational cohomology of $\GL_n(\bfZ)$, which is not known to hold.
\end{nrem}

\subsection*{Stability of Serre filtrations} We prove \cref{mainthm} as a special case of a stronger, but more technical result concerning the Serre filtrations associated to the maps\begin{equation}\label{mapstoarithmeticgroups}\BhAutuo_\partial(W_{g,1})\lra\oB G_g\quad\text{ and }\quad\BlockBDiffuo_\partial(W_{g,1})\lra\oB G_g\end{equation} induced by the natural actions $\hAutuo_\partial(W_{g,1})\ra\GL(\oH_n(W_{g,1}))$ and $\BlockDiffuo_\partial(W_{g,1})\ra\GL(\oH_n(W_{g,1}))$ on the middle homology, whose images $G_g\subset \GL(\oH_n(W_{g,1}))$ turn out to agree (see e.g.\,\cite[Ex.\,5.5]{BerglundMadsen}). Let us fix some notation in order to state the result. For a map $X\ra Y$ to a connected space $Y$, we denote the induced Serre filtrations of $\oH_*(X)$ and $\oH^*(X)$ by \[0\subset F^0\oH_*(X)\subset \ldots\subset F^*\oH_*(X)=\oH_*(X)\quad\text{ and }\quad\oH^*(X)=F_0\oH^*(X)\supset \ldots\supset F_*\oH^*(X)\supset0.\] In particular, $F^0\oH_*(X)$ is the image of the map on homology $\oH_*(F)\ra\oH_*(X)$ induced by the homotopy fibre $F\ra X$ and $F_1\oH^*(X)$ is the kernel of $\oH^*(X)\ra\oH^*(F)$. These filtrations are natural in commutative squares, and as the stabilisation maps for $\BhAutuo_\partial(W_{g,1})$ and $\BlockBDiffuo_\partial(W_{g,1})$ cover a map $\oB G_g\ra \oB G_{g+1}$ induced by the inclusion $W_{g,1}\subset W_{g+1,1}$ (see \cref{arithmeticgroups} below), the Serre filtrations associated to the maps \eqref{mapstoarithmeticgroups} are preserved by the stabilisation maps considered in \cref{mainthm}.

\begin{bigthm}\label{betterthm}
Let $g\ge2$ and set $c=1$ if $n$ is odd and $c=0$ if $n$ is even. For $2n\ge4$, the stabilisation map between the filtration steps of the Serre filtration induced by \eqref{mapstoarithmeticgroups},
\[F_p\oH_*(\BhAutWg;\bfQ)\lra F_p\oH_*(\BhAutWgone;\bfQ),\] is an isomorphism for $p\le g-2+c$ and $*\le ({2n}/{3})(2-c) g-1$. For $2n\ge 6$, the analogous map \[F_p\oH_*(\BlockBDiffWg;\bfQ)\lra F_p\oH_*(\BlockBDiffWgone;\bfQ)\] is an isomorphism for $p\le g-2+c$ and $*\le 2k_n(2-c) g-1$, where 
\[k_n=\begin{cases}
4&\mbox{if }n\equiv 0\Mod{4}\text{ and }n\neq4,8\\
3&\mbox{if }n\equiv 1\Mod{4}\text{ and }n\neq5\\
2&\mbox{if }n\equiv 2\Mod{4}\text{ or }n=8\\
1&\mbox{if }n\equiv 3\Mod{4}\text{ or }n=4,5.\\
\end{cases}\]
\end{bigthm}

\cref{betterthm} includes \cref{mainthm} as the case $p=*$, but shows additionally that the subspaces of low filtration degree stabilise considerably faster than the full homology. This is useful to detect cohomology classes of known filtration degree, as we shall exemplify in the following.

\subsection*{Detecting $\kappa$-classes} It was shown in \cite{BerglundMadsen} that in a range of total degrees growing with $g$, the Serre spectral sequences of the two maps \eqref{mapstoarithmeticgroups} have a product structure and collapse on the $E_2$-page, which implies that the first filtration steps $F_1\oH^*(\BhAutuo_\partial(W_{\infty,1});\bfQ)$ and $F_1\oH^*(\BlockBDiffuo_\partial(W_{\infty,1});\bfQ)$ of the stable cohomology rings coincide with the ideal generated by the image of $\oH^{*>0}(\oB G_\infty;\bfQ)$. By work of Borel \cite{BorelI}, the cohomology $\oH^*(\oB G_\infty;\bfQ)$ is polynomial in generators $\sigma_{4i}$ of degree $4i>0$ if $n$ is even and $\sigma_{4i-2}$ of degree $4i-2>0$ if $n$ is odd. As elucidated in \cite{RWnote}, there are choices of such generators that measure the signature of total spaces of fibrations (or block bundles) with trivialised boundaries and fibre $W_{g,1}$ when pulled back to $\BhAutuo_\partial(W_{g,1})$ (or $\BlockBDiffuo_\partial(W_{g,1})$). In $\oH^*(\BlockBDiffuo_\partial(W_{g,1});\bfQ)$, these classes agree with the generalised Miller--Morita--Mumford (or $\kappa$-) classes $\kappa_{\cL_{i}}$ associated to Hirzebruch's L-classes $\cL_{i}\in\oH^{4i}(\BO;\bfQ)$ for $4i-2n>0$ as defined in \cite{EbertRWBlock}, which in turn coincide with the usually considered $\kappa$-classes $\kappa_{\cL_{i}}$ when pulled back further to $\oH^*(\BDiffuo_\partial(W_{g,1});\bfQ)$ along the usual map $\BDiffuo_\partial(W_{g,1})\ra\BlockBDiffuo_\partial(W_{g,1})$. Specialised to $p=0$ and taken duals, \cref{betterthm} thus reads as follows. 

\begin{bigcor}\label{secondthm}
Let $g\ge2$ and set $c=1$ if $n$ is odd and $c=0$ if $n$ is even. For $2n\ge4$, the natural map
\[\frac{\oH^*(\BhAutuo_\partial(W_{\infty,1});\bfQ)}{(\sigma_{4i-2c}\mid i>0)}\lra \frac{\oH^*(\BhAutuo_\partial(W_{g,1});\bfQ)}{F_1\oH^*(\BhAutuo_\partial(W_{g,1});\bfQ)}\]
is an isomorphism for $*\le ({2n}/{3})(2-c)g-1$. If $2n\ge6$, then the map
\[\frac{\oH^*(\BlockBDiffuo_\partial(W_{\infty,1});\bfQ)}{(\sigma_{4i-2c}\mid i>0)}\lra \frac{\oH^*(\BlockBDiffuo_\partial(W_{g,1});\bfQ)}{F_1\oH^*(\BlockBDiffuo_\partial(W_{g,1});\bfQ)}\] 
is an isomorphism for $*\le 2k_n(2-c) g-1$, where $k_n$ is defined as in \cref{betterthm}.
\end{bigcor}

Using the mentioned collapse of the Serre spectral sequences of \eqref{mapstoarithmeticgroups} in a range, Berglund--Madsen \cite{BerglundMadsen} identified the stable cohomology rings $\oH^*(\BhAutuo_\partial(W_{\infty,1});\bfQ)$ and $\oH^*(\BlockBDiffuo_\partial(W_{\infty,1});\bfQ)$ abstractly in terms of $\oH^*(\oB G_\infty;\bfQ)$ and the homology of certain graph complexes. The ring $\oH^*(\BDiffuo_\partial(W_{\infty,1});\bfQ)$ on the other hand was computed by Galatius--Randal-Williams \cite{GRWstable} as the polynomial algebra on $\kappa$-classes $\{\kappa_c\mid c\in\cB_{>2n}\}$, where $\cB$ is the set of monomials in the Euler class $e$ of degree $2n$ and the Pontryagin classes $p_i$ of degree $4i$ for $i=\lceil\frac{n+1}{4}\rceil,\ldots,n-2,n-1$. The map on rational cohomology induced by $\BDiffuo_\partial(W_{g,1})\ra\BlockBDiffuo_\partial(W_{g,1})$ hence exhibits the morphism \[\frac{\bfQ[\kappa_{c}\mid c\in\cB_{>2n}]}{(\kappa_{\cL_i}\mid 4i-2n>0)}\lra \frac{\oH^*(\BlockBDiffuo_\partial(W_{\infty,1});\bfQ)}{(\kappa_{\cL_i}\mid 4i-2n>0)}\] as a retract, so \cref{secondthm} shows that every nontrivial class in the left hand quotient (for instance $\kappa_c$ for a decomposable monomial $c\in\cB_{>2n}$) is nontrivial in $\oH^*(\BlockBDiffuo_\partial(W_{g,1});\bfQ)$ for significantly smaller values of $g$ than predicted by \cref{mainthm}.

\subsection*{Diffeomorphism groups}Extending (block) diffeomorphism of an embedded disc $D^{2n}\subset W_{g,1}$ by the identity induces a commutative square
\begin{center}
\begin{tikzcd}
\BDiffuo_\partial(D^{2n})\arrow[r]\arrow[d]&\BDiffWg\arrow[d]\\
\BlockBDiffuo_\partial(D^{2n})\arrow[r]&\BlockBDiffWg
\end{tikzcd}
\end{center}
whose induced map on vertical homotopy fibres \[\widetilde{\Diffuo}_\partial(D^{2n})/\Diffuo_\partial(D^{2n})\lra \BlockDiffWg/\DiffWg\] is $(2n-4)$-connected by Morlet's lemma of disjunction \cite[Cor.\,3.2]{BurgheleaLashofRothenberg}. As noted by Randal-Williams \cite[Sect.\,4]{UpperBound}, this can be exploited for $g\gg0$ to conclude that $\widetilde{\Diffuo}_\partial(D^{2n})/\Diffuo_\partial(D^{2n})$ is rationally $(2n-5)$-connected as a consequence of Berglund--Madsen's work. Applying Morlet's lemma yet another time shows that $ \BlockDiffWg/\DiffWg$ as rationally $(2n-5)$-connected \emph{for all $g\ge0$}, so the map $\oH_*(\BDiffWg;\bfQ)\ra\oH_*(\BlockBDiffWg;\bfQ)$ is an isomorphism for $*\le2n-5$. This shows that in this range of degrees, our results for $\oH_*(\BlockBDiffuo_\partial(W_{g,1});\bfQ)$ are equally valid for $\oH_*(\BDiffuo_\partial(W_{g,1});\bfQ)$, which in particular implies that the latter groups stabilise in this range much quicker than what is known from Galatius--Randal-Williams' work \cite{GRWI}.

\begin{bigcor}For $2n\ge6$, the stabilisation map 
\[\oH_*(\BDiffWg;\bfQ)\lra\oH_*(\BDiffWgone;\bfQ)\]
is an isomorphism for $*\le \min\{g-2+c,2n-5\}$, where $c=1$ if $n$ is odd and $c=0$ if $n$ is even.
\end{bigcor}

\subsection*{Acknowledgements}
I would like to thank Oscar Randal-Williams for helpful discussions and comments on a draft of this note, and Alexander Kupers for his comments. I was supported by the European Research Council (ERC) under the European Union’s Horizon 2020 research and innovation programme (grant agreement No. 756444).

\numberwithin{equation}{section}

\section{Arithmetic groups and mapping class groups}\label{arithmeticgroups}
Following \cite[Sect.\,2.1]{KRW}, we fix $\epsilon=\pm1$ and a nondegenerate $\epsilon$-symmetric bilinear form \[\lambda\colon H(g)\otimes H(g)\lra\bfQ\] on a rational vector space $H(g)$ of dimension $2g$, required to be of signature $0$ if $\epsilon=1$. The automorphisms of $H(g)$ preserving the form $\lambda$, commonly denoted by $\Sp_{2g}(\bfQ)$ for $\epsilon=-1$ and by $\oO_{g,g}(\bfQ)$ for $\epsilon=1$, form the rational points of an algebraic group $\bfG_g\in\{\bfSp_{2g},\bfO_{g,g}\}$. We call a subgroup $G\le\bfG_g(\bfQ)$  \emph{arithmetic} if it is commensurable to $\bfG_g(\bfZ)$ and not entirely contained in the index $2$ subgroup $\SO_{g,g}(\bfQ)\le \oO_{g,g}(\bfQ)$ if $\bfG_g=\bfO_{g,g}$.\footnote{This condition is non-standard and serves to ensure that $G$ is Zariski dense in $\bfG_g(\bfQ)$ (see \cite[Sect.\,2.1.1]{KRW}).} Defining $H(g+1)=H(g)\oplus H(1)$ after having fixed $H(1)$, there are natural inclusion maps $\bfG_g(\bfQ)\ra\bfG_{g+1}(\bfQ)$ and we denote their colimit by $\bfG_\infty(\bfQ)$.

The modification of Berglund--Madsen's argument we shall present is mainly an application of the following homological vanishing result for arithmetic groups, which can be derived from a combination of a result of Margulis \cite{Margulis} with work of Borel \cite{BorelI,BorelII}, improved in range by Tshishiku \cite{TshishikuRange} (cf.\,\cite[Thm\,2.2, 2.3]{KRW} and \cite[Sect.\,3]{EbertRW}).

\begin{thm}\label{Borel}Set $c=1$ if $\epsilon=-1$ and $c=0$ if $\epsilon=1$. There is a map \[\oB\bfG_\infty(\bfQ)\lra\Omega^{\infty+2c}_0\KO\] such that for $g\ge4-c$ and any finite-dimensional $\bfQ[G]$-module $V$ for an arithmetic group $G\le \bfG_g(\bfQ)$, the natural maps
\[\oH_*(G;V)\xlra{\Delta} \oH_*(G;\bfQ)\otimes V_{G}\lra \oH_*(\Omega^{\infty+2c}_0\KO;\bfQ)\otimes V_{\bfG_g(\bfQ)}\] 
are isomorphisms for $*\le g-2+c$.
\end{thm}

One immediate consequence of \cref{Borel} is that for an arithmetic subgroup $G\le\bfG_g(\bfQ)$, taking coinvariants is an exact functor on the category of finite dimensional $\bfQ[G]$-modules as long as $g\ge4-c$, because $\oH_1(G;-)$ vanishes on such modules since $\Omega^{\infty+2c}_0\KO$ is rationally $1$-connected. In fact, this vanishing of $\oH_1(G;-)$ holds already for $g\ge2$ (cf.\,\cite[Thm A.1]{BerglundMadsen} or \cite[Sect.\,2.1.2]{KRW}).

\begin{nrem}\ 
\begin{enumerate}
\item Note that the range in \cref{Borel} does not depend on the representation $V$; this is Thishiku's \cite{TshishikuRange} improvement of original Borel's result. The proof of \cref{betterthm} crucially relies on this improvement, whereas the weaker \cref{mainthm} could be proved using a coarser estimate of the range that depends on the representation, such as in \cite[Prop.\,3.7]{EbertRW}.
\item Other work of Tshishiku \cite{Tshishiku} shows that the cohomology $\oH^g(\Gamma;\bfQ)$ of arithmetic subgroups $\Gamma\le\oO_{g,g}(\bfQ)$ in degree $g$ can be arbitrarily large when $g$ is odd, which illustrates that the range in \cref{Borel} is for $\epsilon=1$ close to being optimal.
\end{enumerate}
\end{nrem}

The example of a form $\lambda\colon H(g)\otimes H(g)\ra\bfQ$ we have in mind is the $(-1)^n$-symmetric intersection form on the middle homology $H(g)=\oH_n(W_{g,1};\bfQ)$ of the manifolds $W_{g,1}=D^{2n}\sharp(S^n\times S^n)^{\sharp g}$. Occasionally, it will be convenient to view these manifolds alternatively as the iterated boundary connected sum $W_{g,1}=\natural^g W_{1,1}$ of $W_{1,1}=S^n\times S^n\backslash\interior{D^{2n}}$. The evident action of the homotopy mapping class group $\pi_0\hAutuo_\partial(W_{g,1})$ on $H(g)$ preserves $\lambda$ and hence gives rise to a morphism
\begin{equation}\label{homologyaction}\pi_0\hAutWg\lra \bfG_g(\bfQ)\end{equation} whose kernel turns out to be finite (see e.g.\,\cite[Prop.\,5.3]{BerglundMadsen}). Its image
\[G_g\coloneq\im\big(\pi_0\hAutWg\ra \bfG_g(\bfQ)\big)\] is an arithmetic subgroup (see e.g.\,\cite[p.\,29]{KRW}) which agrees with the image of the block mapping class group $\pi_0\BlockDiffWg$ in $\bfG_g(\bfQ)$ (see e.g.\,\cite[Ex.\,5.5]{BerglundMadsen}). Stabilising by $(-)\oplus H(1)$ as described above corresponds to taking the boundary connected sum with a further copy of $W_{1,1}$, so the inclusion $\bfG_g(\bfQ)\ra\bfG_{g+1}(\bfQ)$ is covered by the stabilisation map $\hAutuo_\partial(W_{g,1})\ra \hAutuo_\partial(W_{g+1,1})$ induced by extending homotopy equivalences by the identity.

\section{The argument of Berglund and Madsen}\label{BMargument}

In what follows, we sketch the proof of Berglund--Madsen's homology stability result for the spaces $\BhAutuo_\partial(W_{g,1})$ and $\BlockBDiffuo_\partial(W_{g,1})$ as presented in \cite{BerglundMadsen}. For $\BhAutuo_\partial(W_{g,1})$, one first shows that the classifying space of the component of the identity $\hAutuo^{\id}_\partial(W_{g,1})\subset\hAutuo_\partial(W_{g,1})$ is coformal, i.e.\,that its Quillen dg Lie algebra is formal. The homotopy Lie algebra \[\frg_g\coloneq \pi_{*+1}\BhAutuo^{\id}_\partial(W_{g,1})\otimes\bfQ\cong \Der^+_{\omega}\big(\bfL \big(H(g)[n-1] \big)\big)\] can be identified with the graded Lie algebra of those derivations of positive degree of the free Lie algebra on the graded vector space $H(g)=\oH_n(W_{g,1};\bfQ)$, concentrated in degree $(n-1)$, that vanish on the dual $\omega\in \bfL(H(g)[n-1])$ of the intersection form $\lambda$ (see \cite[Sect.\,3.5, Cor.\,3.13]{BerglundMadsen}). The coformality ensures that the Quillen spectral sequence
\[E_{p,q}^2=\oH^{\CE}_{p,q}(\frg_g)\implies \oH_{p+q}(\BhAutuo^{\id}_\partial(W_{g,1}))\]
collapses, so there is a (non-canonical) isomorphism between the rational homology of $\BhAutuo^{\id}_\partial(W_{g,1})$ and the Chevalley--Eilenberg homology of $\frg_g$, \begin{equation}\label{homologyhaut}\oH_*(\BhAutuo^{\id}_\partial(W_{g,1});\bfQ)\cong \oH^{\CE}_*(\frg_g).\end{equation} 
The isomorphism \eqref{homologyhaut} can be chosen to be compatible with the evident stabilisation maps and the natural action of $\pi_0\hAutuo_{\partial}(W_{g,1})$ on both sides, the one on the right being induced from the action on $H(g)$ via \eqref{homologyaction} by functoriality (see \cite[Prop.\,7.9]{BerglundMadsen}). This uses the fact that for $g\ge2$, every extension of finite dimensional $\bfQ[\pi_0\hAutuo_{\partial}(W_{g,1})]$-modules splits, since $\pi_0\hAutuo_{\partial}(W_{g,1})\ra G_g$ has finite kernel and $G_g$ has this property. Therefore, using finiteness of the latter kernel once again, the $E_2$-page of the Serre spectral sequence of the fibration sequence
\[\BhAutuo^{\id}_{\partial}(W_{g,1})\lra\BhAutuo_{\partial}(W_{g,1})\lra\oB\pi_0\hAutuo_{\partial}(W_{g,1})\] is isomorphic to $\oH_*(G_g;\oH_*^{\CE}(\frg_g))$, so by a spectral sequence comparison, it suffices to show that \begin{equation}\label{E2pagestabilisation}\oH_*(G_g;\oH_*^{\CE}(\frg_g))\lra \oH_*(G_{g+1};\oH_*^{\CE}(\frg_{g+1}))\end{equation} induces an iso- and epimorphism in a suitable range of total degrees. One way of proving this is to show that the Chevalley--Eilenberg chains $C_*^{\CE}(\frg_g)$ in a fixed total degree are the value on $H(g)$ of a \emph{Schur functor} $P$ associated to a symmetric sequence $\{P(m)\}_{m\ge0}$, that is, an endofunctor on the category of $\bfQ$-vector spaces of the form
\[\textstyle{P(V)=\bigoplus_{m=0}^kP(m)\otimes_{\Sigma_m}V^{\otimes m}}\] for a sequence of $\bfQ$-vector spaces $P(m)$ equipped with an action of the symmetric groups $\Sigma_m$. The minimal choice of $k$ is called the \emph{degree} of $P$, which we denote by $\deg(P)$.

\begin{lem}[Berglund--Madsen {\cite[Lem.\,7.2]{BerglundMadsen}}]\label{CEchains1}For $n\ge2$, the Chevalley--Eilenberg chains $C^{\CE}_q(\frg_g)$ in total degree $q$ are isomorphic to the value of a Schur functor of degree $\le 3q/n$ on $H(g)$.
\end{lem}

This identification of $C_*^{\CE}(\frg_g)$ is tailored to invoke a homological stability result due to Charney \cite{Charney}, which implies that $\oH_*(G_g;C_*^{\CE}(\frg_g))\ra \oH_*(G_{g+1};C_*^{\CE}(\frg_{g+1}))$ is an isomorphism in a range of total degrees with slope $1/2$. From this, the desired property of \eqref{E2pagestabilisation} can be concluded from a hypercohomology argument (see \cite[Prop.\,7.11]{BerglundMadsen}).

In the case of block diffeomorphisms, one proceeds similarly: the classifying space $\BlockBDiffuo_{\partial,\circ}(W_{g,1})$ of the kernel of the map $\BlockDiffuo_\partial(W_{g,1})\ra\pi_0\hAutuo_\partial(W_{g,1})$ is coformal and its homotopy Lie algebra
\[\pi_{*+1}\BlockBDiffuo_{\partial,\circ}(W_{g,1})\otimes \bfQ\cong \frg_g\oplus \fra_g\]
splits as a sum of the graded Lie algebra $\frg_g$ defined above with the graded abelian Lie algebra \[\fra_g\coloneq (H(g)\otimes \Pi[-n])_{\ge0},\] where $\Pi=\pi_*(\Omega\BO)\otimes\bfQ$ (see \cite[Cor.\,4.26]{BerglundMadsen}). By an argument analogous to the one for homotopy equivalences, it suffices to show that the stabilisation map \[\oH_*(G_g;\oH_*^{\CE}(\frg_g\oplus\fra_g))\lra \oH_*(G_{g+1};\oH_*^{\CE}(\frg_{g+1}\oplus\fra_{g+1}))\] is an isomorphism in the wished range of degrees, which is accomplished by proving that $C^{\CE}_q(\frg_g\oplus\fra_g)$ is the value of a Schur functor of degree $\le2q$ on $H(g)$ and finishing as before (see \cite[Prop.\,7.17]{BerglundMadsen}).

As we shall explain in the next section, it is this last step in the argument which we modify to achieve the better stability ranges stated in \cref{betterthm}, i.e.\,the use of Charney's result followed by an argument involving hypercohomology.

\section{Stability of Serre filtrations}\label{modification}
By the discussion of the previous section, the action of $\pi_0\hAutuo_\partial(W_{g,1})$ on $\oH_*(\BhAutuo_\partial^{\id}(W_{g,1});\bfQ)$ factors through the natural morphism $\pi_0\hAutuo_\partial(W_{g,1})\ra G_g$. As the latter has finite kernel, we conclude that the Serre spectral sequence whose Serre filtration \cref{betterthm} is concerned with is canonically isomorphic to the spectral sequence considered in \cref{BMargument}, i.e.\,the one induced by the map $\BhAutuo_\partial(W_{g,1})\ra\oB \pi_0\hAutuo_\partial(W_{g,1})$.  In particular, there is an isomorphism between its $E_2$-page and $\oH_*(G_g;\oH_*^{\CE}(\frg_g))$, compatible with the stabilisation map $\oH_p(G_g;\oH_q^{\CE}(\frg_g))\ra \oH_p(G_{g+1};\oH_q^{\CE}(\frg_{g+1}))$. It will follow from \cref{stabilisationiso} below that this stabilisation map is an isomorphism in the range $p\le g-2+c$ and $p+q\le (2n/3)(2-c)g-1$, which would imply the claim of \cref{betterthm} regarding $\BhAutuo_\partial(W_{g,1})$ if all incoming and outgoing differentials in this range of bidegrees were trivial. But Berglund--Madsen \cite[Sect.\,8]{BerglundMadsen} showed that the spectral sequence collapses on the $E_2$-page in a range of total degrees growing with $g$, so all differentials whose target is already in the stability range have to vanish, even though their source might not be in the stable range. Outgoing differentials in the above specified range vanish for the same reason, since this range is closed under differentials.

\begin{lem}\label{stabilisationiso}
Set $c=1$ if $\epsilon=-1$ and $c=0$ if $\epsilon=1$ is even. For $g\ge2$, the stabilisation map
\[\oH_p(G_g;\oH_q^{\CE}(\frg_g))\lra\oH_p(G_{g+1};\oH_q^{\CE}(\frg_{g+1}))\]
is an isomorphism for $p\le g-2+c$ and $q\le (2n/3)(2-c)g-1$.
\end{lem}
\begin{proof}
The claim for $p=1$ follows from the fact that $\oH_1(G_g;-)$ vanishes on finite dimensional $\bfQ[G_g]$-modules for $g\ge2$. For the same reason, taking $G_g$-coinvariants of such modules is exact, so to establish the case $p=0$, it is enough to show that the stabilisation map on the chain level $C_q^{\CE}(\frg_g)_{\bfG_g(\bfQ)}\ra C_q^{\CE}(\frg_{g+1})_{\bfG_{g+1}(\bfQ)}$ is an isomorphism for $q\le (2n/3)(2-c)g$. By \cref{CEchains1}, the chains $C_q^{\CE}(\frg_g)$ in total degree $q$ are the value of a Schur functor of degree $\le 3q/n$, compatible with the stabilisation map (see \cite[Sect.\,7]{BerglundMadsen}), so the claim is implied by \cref{coinvariantiso} below. For $p\ge2$ and hence $g\ge4-c$, the map in consideration fits into a commutative square
\begin{center}
\begin{tikzcd}
\oH_p(G_g;\oH_q^{\CE}(\frg_g))\arrow[r]\arrow[d]&\oH_p(G_{g+1};\oH_q^{\CE}(\frg_{g+1}))\arrow[d]\\
\oH_p(\Omega^{\infty+2c}_0\KO;\bfQ)\otimes \oH_q^{\CE}(\frg_g)_{\bfG_g(\bfQ)}\arrow[r]&\oH_p(\Omega^{\infty+2c}_0\KO;\bfQ)\otimes \oH_q^{\CE}(\frg_{g+1})_{\bfG_{g+1}(\bfQ)}
\end{tikzcd}
\end{center} whose vertical arrows are isomorphisms for $p\le g-2+c$ by \cref{Borel}, which reduces the claim for $p\ge2$ to the case $p=0$ already proved.
\end{proof}

\begin{lem}\label{coinvariantiso}For a Schur functor $P$ of degree $\le k$ and $c$ as in \cref{stabilisationiso}, the morphism 
\[P(H(g))_{\bfG_g(\bfQ)}\lra P(H(g+1))_{\bfG_{g+1}(\bfQ)}\] is an isomorphism for $k\le 2(2-c)g$ and a monomorphism for $k\le 2(2-c)(g+1)$.
\end{lem}
\begin{proof}We can assume that $P$ is of the form $P(V)=P(k)\otimes_{\Sigma_k}V^{\otimes k}$. Writing
\[\big(P(k)\otimes_{\Sigma_k}H(g)^{\otimes k}\big)_{\bfG_g(\bfQ)}\cong \big((P(k)\otimes H(g)^{\otimes k})_{\Sigma_k}\big)_{\bfG_g(\bfQ)}\cong \big(P(k)\otimes H(g)^{\otimes k}_{\bfG_g(\bfQ)}\big)_{\Sigma_k},\] we see that it suffices to show that the morphism \[H(g)^{\otimes k}_{\bfG_g(\bfQ)}\lra H(g+1)^{\otimes k}_{\bfG_{g+1}(\bfQ)}\] has the asserted quality, or dually that 
\[\big(H(g+1)^{\otimes k}\big)^{\bfG_{g+1}(\bfQ)}\lra \big(H(g)^{\otimes k}\big)^{\bfG_g(\bfQ)}\] is an isomorphism for $k\le 2(2-c)g$ and an epimorphism for $k\le 2(2-c)(g+1)$. This follows from the first and second fundamental theorem of invariant theory for $\Sp_{2g}(\bfC)$ and $\oO_{g,g}(\bfC)\cong\oO_{2g}(\bfC)$ by applying $(-)\otimes_\bfQ\bfC$ and using Zariski density of $\bfG_g(\bfQ)\subset \bfG_g(\bfC)$ (c.f.\,\cite[Thm\,2.6, Sect.\,9.4]{KRW}).
\end{proof}

This finishes the proof of \cref{betterthm} for homotopy equivalences. The case of block diffeomorphisms is proved analogously by showing that the stabilisation map
\[\oH_p(G_g;\oH_q^{\CE}(\frg_g\oplus\fra_g))\lra\oH_p(G_{g+1};\oH_q^{\CE}(\frg_{g+1}\oplus\fra_{g+1}))\] is an isomorphism for $p\le g-2+c$ and $p+q\le 2k_n(2-c)g-1$, which can be derived by adapting the argument we gave to prove \cref{stabilisationiso}, using the following refinement of Berglund--Madsen's estimate \cite[Prop.\,7.17]{BerglundMadsen} on the degree of the relevant Schur functor.

\begin{lem}\label{CEchains2}For $n\ge3$, the Chevalley--Eilenberg chains $C^{\CE}_q(\frg_g\oplus \fra_g)$ in total degree $q$ are isomorphic to the value on $H(g)$ of a Schur functor of degree \[\deg(P)\le\begin{cases}
  q/4&\mbox{if }n\equiv 0\Mod{4}\text{ and }n\ge12\\
 q/3&\mbox{if }n\equiv 1\Mod{4}\text{ and }n\ge9\\
  q/2&\mbox{if }n\equiv 2\Mod{4}\text{ or }n=8.\\
  q&\mbox{if }n\equiv 3\Mod{4}\text{ or }n=4,5.\\
\end{cases}\]
\end{lem}
\begin{proof}
The proof of \cite[Prop.\,7.17]{BerglundMadsen} shows that there is a sequence of graded $\bfQ[\Sigma_m]$-modules $P(m)$ for $m\ge0$, supported in degrees $\ge(n/3)m$, that participates in a graded isomorphism 
\[\textstyle{C^{\CE}_*(\frg_g\oplus \fra_g)\cong \bigoplus_{m\ge0}(P\otimes Q)(m)\otimes_{\Sigma_m}H(g)^{\otimes m}}\] of $G_g$-modules, where $Q(m)=(\pi_*(\BO)[-n]\otimes\bfQ)_{>0}^{\otimes m}$ and
\[\textstyle{(P\otimes Q)(m)=\bigoplus_{k+l=m}\operatorname{Ind}^{\Sigma_m}_{\Sigma_k\times \Sigma_l}P(k)\otimes Q(l)}.\] The graded module $P(k)\otimes Q(l)$ is supported in degrees $\ge (n/3)k+l\ge k+l$, so the degree $q$-part of the Schur functor associated to $\{(P\otimes Q)(m)\}_{m\ge0}$ is at most of degree $\le q$, which proves the bottommost case of the statement. To obtain the improved estimate on the degree in the other cases, note that since $\pi_*(\BO)\otimes\bfQ$ is only nontrivial in degrees $*\equiv0\Mod{4}$, the $\Sigma_{k+l}$-module $P(k)\otimes Q(l)$ is in fact supported in degrees $\ge(n/3)k+4l$ if $n\equiv 0\Mod{4}$, in degrees $\ge(n/3)k+3l$ if $n\equiv 1\Mod{4}$, and in degrees $\ge(n/3)k+2l$ if $n\equiv 2\Mod{4}$. A careful case distinction finishes the proof.
\end{proof}

\bibliographystyle{amsalpha}
\bibliography{literature}
\vspace{0.2cm}
\end{document}